\documentclass[11pt,a4paper]{amsart}
\usepackage{amsmath,amssymb,amsthm,amscd,verbatim,enumerate}
\usepackage{graphicx}
\usepackage{hyperref}

\newcommand{\spacing}[1]{
\renewcommand{\baselinestretch}{#1}
\setlength{\footnotesep}{\baselinestretch\footnotesep}}
\spacing{1.1}

\DeclareMathOperator{\up}{up}

\DeclareMathOperator{\pw}{pw}
\DeclareMathOperator{\FF}{FF}

\renewcommand{\l}{\ell}  

\newtheoremstyle{mystyle}{\parskip}{0mm}{\itshape}{}{\bf}{.}{.5em}{}
\renewenvironment{proof}[1][\proofname]{{\itshape #1. }}{\qed\\}

\theoremstyle{mystyle}
\newtheorem{theorem}{Theorem}[section]

\newtheorem{corollary}[theorem]{Corollary}

\newtheorem{claim}[theorem]{Claim}

\theoremstyle{definition}

\newcommand{\kk}[1]{$\mathbf{#1 + #1}$}

\renewcommand{\leq}{\leqslant}
\renewcommand{\geq}{\geqslant}

\newcommand{\DEF}{\sl}

\newcommand{\mc}{\mathcal}

\newcommand{\fig}[4]{
\begin{figure}
\centering
\includegraphics[width=#1\textwidth]{#2}
\caption{\label{#4}#3}
\end{figure}
}

\begin{document}

\title[First-Fit on Posets Without Two Long Incomparable Chains]{An Improved Bound for First-Fit on Posets\\ 
Without Two Long Incomparable Chains}

\author{Vida Dujmovi\'c}
\address{\newline School of Computer Science
\newline Carleton University
\newline Ottawa, Canada}
\email{vida@scs.carleton.ca}

\author{Gwena\"el Joret}
\address{\newline  D\'epartement d'Informatique 
\newline Universit\'e Libre de Bruxelles
\newline Brussels, Belgium}
\email{gjoret@ulb.ac.be}

\author{David~R.~Wood}
\address{\newline Department of Mathematics and Statistics
\newline The University of Melbourne
\newline Melbourne, Australia}
\email{woodd@unimelb.edu.au}

\thanks{
Vida Dujmovi\'c is supported by the Natural Sciences and Engineering Research Council (NSERC) of Canada,
and by an Endeavour Fellowship from the Australian Government.  
Gwena\"el Joret is a Postdoctoral Researcher of the Fonds
 National de la Recherche Scientifique (F.R.S.--FNRS), and is also
 supported by an Endeavour Fellowship from the Australian Government.
 David Wood is supported by a QEII Research Fellowship from the
 Australian Research Council}

\date{\today}
\sloppy
\maketitle

\begin{abstract}
It is known that the First-Fit algorithm for partitioning a poset $P$ into chains
uses relatively few chains when $P$ does not have two incomparable chains each of size $k$. 
In particular, if $P$ has width $w$ then
Bosek, Krawczyk, and Szczypka ({\it SIAM J. Discrete Math.}, 23(4):1992--1999, 2010)
proved an upper bound of $ckw^{2}$ on the number of chains used by First-Fit for some constant $c$, 
while Joret and Milans ({\it Order}, 28(3):455--464, 2011) gave one of $ck^{2}w$. 
In this paper we prove an upper bound of the form $ckw$.   
This is best possible up to the value of $c$.  
\end{abstract}

\section{Introduction}

For a poset $P$, the maximum size of an antichain in $P$ is called the 
{\DEF width} of $P$. Every partition of $P$ into chains contains at least $w$ chains, where $w$
is the width of $P$.  
By a classical theorem of Dilworth~\cite{D50}, there always exists a chain partition of $P$  
achieving this lower bound. While such an optimal chain partition can 
easily be computed (see for instance~\cite{S03A}), 
this computation requires a full knowledge of the poset $P$ and cannot be 
made {\em on-line}: In the on-line setting, elements of $P$ are uncovered one at a time
and a chain decomposition of the poset uncovered so far must be maintained
at all times. In this model, once an element is assigned to some 
chain it must remain assigned to that chain during the whole execution. 

Szemer\'edi proved
that every on-line algorithm can be forced to use $\Omega(w^{2})$ chains   
(see~\cite{BFKKMM, K_CM}). 
It is a well-known open problem to decide whether there exists
an on-line algorithm that uses polynomially many chains (in $w$). 
The current best bound, 
due to Bosek and Krawczyk~\cite{BK_FOCS},  
is sub-exponential: $w^{16\log_{2}w}$. 

First-Fit is a natural on-line algorithm for partitioning
a poset $P$ into chains: 
Each time a new element $v$ is uncovered, First-Fit puts $v$ into the first chain
in the current chain partition such that $v$ is comparable to all elements in
that chain. If no such chain is found, then a new chain containing only $v$ is added
at the end of the current chain partition. 

The performance of First-Fit on various classes
of posets has been studied extensively~\cite{
BFKKMM, BKM_eurocomb, BKS_SIDMA, BKT_mscript, CS_RAIRO, FKT_sub, K_CM,  
KQ_DM, S_PhD, NS_ORDER, PRV_SODA, PRV_TALG}. 
In particular, 
Kierstead~\cite{K_CM} showed that First-Fit can be forced to use an unbounded 
number of chains even on posets of width $2$.  Nevertheless, First-Fit
behaves well on some restricted classes of posets. A prominent example 
are interval orders, for which   
Kierstead~\cite{K_SIDMA} obtained a linear bound of $40w$,
which was subsequently improved by Kierstead and Qin~\cite{KQ_DM} to $25.8w$, 
and then by Pemmaraju, Raman and Varadarajan~\cite{PRV_SODA, PRV_TALG} to $10w$.  
It was later shown by Brightwell, Kierstead and Trotter~\cite{BKT_mscript} and
by Narayanaswamy and Babu~\cite{NS_ORDER} that the proof method of Pemmaraju 
{\it et al.\ }actually gives a bound of $8w$. (This refined analysis is also
presented in the journal version~\cite{PRV_TALG} of~\cite{PRV_SODA}.)   
As for lower bounds, the best result to date is that, for every $\varepsilon > 0$, 
First-Fit can be forced to use at least $(5 - \varepsilon)w$ chains provided $w$ 
is large enough~\cite{S_PhD}.

By a well-known theorem of Fishburn~\cite{F_JMP}, interval orders are exactly
the posets not containing {\kk{2}} as an induced subposet, where {\kk{k}} 
denotes the poset consisting of two disjoint chains $A,B$ with $|A|=|B|=k$ 
where every element in $A$ is incomparable with every element in $B$. 
It is therefore natural to ask to which extent the good performance
of First-Fit on interval orders extends to posets without {\kk{k}} where $k \geq 2$.  
This question was first considered by Bosek, Krawczyk and Szczypka~\cite{BKS_SIDMA}, 
who proved an upper bound of $3kw^{2}$ on the number of chains used 
by First-Fit. Joret and Milans~\cite{JM_Order} subsequently showed an upper bound of
$8(k-1)^{2}w$, which is asymptotically better when $k$ is fixed.  
However, the two bounds are incomparable if $k$ and $w$ are independent variables. 

The main result of this paper is that a linear dependency in $k$ and in $w$ 
can be guaranteed simultaneously: For $k\geq2$, First-Fit partitions
every poset of width $w$ without {\kk{k}} into at most $8(2k-3)w$ chains. 
We also give an example where
First-Fit uses $(k-1)(w-1)$ chains on such posets, implying that our
upper bound is within a constant factor of optimal. 

Our proof of the upper bound is comprised of two steps.  
First we prove that the incomparability graph 
of every poset of width $w$ without {\kk{k}} has small pathwidth, 
namely pathwidth at most $(2k-3)w-1$.    
Then we show that the fact  
that First-Fit uses at most $8w$ chains on interval orders of width $w$, 
as proved in~\cite{PRV_TALG},     
implies that First-Fit uses at most $8(p+1)$ chains on posets whose incomparability 
graphs have pathwidth $p$. 
Combining these two results, we obtain an upper bound of $8(2k-3)w$ 
on the number of chains used by First-Fit on posets 
of width $w$ without {\kk{k}}. 

\section{Definitions}
\label{sec:definitions}

A {\DEF chain} (respectively, {\DEF antichain}) in a poset $P$ is a set of pairwise 
comparable (incomparable) elements in $P$. 
The maximum size of an antichain in $P$ is called the {\DEF width} of $P$. 
An element $v$ is {\DEF minimal} (respectively, {\DEF maximal}) in $P$ if there is no element
$w$ in $P$ such that $w < v$ in $P$ ($w > v$ in $P$).  

The {\DEF incomparability graph} of a poset $P$ is the graph with vertex set
the elements of $P$  where two distinct vertices are adjacent 
if and only if the corresponding elements are incomparable in $P$.

A {\DEF First-Fit chain partition} of a poset $P$ is a sequence $C_{1}, \dots, C_{q}$
of non-empty disjoint chains of $P$ such that every element of $P$ is in one of the chains, and
for each $i,j$ such that $1\leq i < j \leq q$ and each $v\in C_{j}$, there exists $w \in C_{i}$ 
such that $v$ and $w$ are incomparable in $P$. 
Observe that every chain partition produced by the First-Fit algorithm is
a First-Fit chain partition, and conversely every First-Fit chain partition
can be produced by First-Fit. 

A {\DEF First-Fit coloring} of a graph $G$  
is a coloring of the vertices of $G$ with positive integers such that 
every vertex $v\in V(G)$ that is colored $i \geq 2$ has a neighbor colored $j$ 
for every $j \in \{1, \dots, i-1\}$. The maximum number of colors in a 
First-Fit coloring of $G$ is denoted $\FF(G)$. 
Note that a First-Fit chain partition of a poset $P$ can equivalently
be seen as a First-Fit coloring of the incomparability graph of $P$. 

Every (finite) set $\mathcal{I}$ of closed intervals of the real line
defines a corresponding poset $P$ as follows: $P$ has one element per interval in 
$\mathcal{I}$, and $u < v$ in $P$ if and only if $I(u) = [a,b]$ and $I(v) = [c, d]$
with $b < c$, where $I(u)$ denotes the interval corresponding to $u$. 
The set $\mathcal{I}$ is said to be an {\DEF interval representation}
of $P$. A poset $P$ is an {\DEF interval order} if and only if $P$ has
an interval representation.

Such a set $\mathcal{I}$ also defines a corresponding graph $G$, namely the intersection
graph of the intervals in $\mathcal{I}$. Thus $G$ has one vertex per interval, 
and two distinct vertices are adjacent if and only if the corresponding intervals 
intersect. Similarly as above, $\mathcal{I}$ is said to be an {\DEF interval representation}
of $G$. A graph is an {\DEF interval graph} if it has an interval representation.   
Clearly, the incomparability graph of an interval order is an interval graph, 
and conversely every interval graph is the incomparability graph of some interval order.  

A {\DEF path decomposition} of a graph $G$ is a sequence $B_{1}, \dots, B_{k}$ 
of vertex subsets of $G$ (called {\DEF bags}) such that each
vertex of $G$ appears in a non-empty consecutive set of bags, and
each edge of $G$ has its two endpoints in at least one bag. 
The {\DEF width} of the decomposition is the maximum cardinality of a bag minus one. 
The {\DEF pathwidth} $\pw(G)$ of $G$ is the minimum width of a path decomposition of $G$.  
Note that the pathwidth of $G$ can equivalently be defined as the minimum integer $k$ 
such that $G$ is a spanning subgraph of an interval graph $H$ with $\omega(H)=k+1$ 
(where $\omega(H)$ denotes the maximum cardinality of a clique in $H$).

\section{Proofs}
\label{sec:proofs}

\subsection{Pathwidth of posets without {\kk{k}}}

A poset $P$ {\DEF extends} (or 
{\DEF is an extension of}) a poset $Q$ if
$P$ and $Q$ have the same set of elements and  
$u < v$ in $Q$ implies $u < v$ in $P$ for all elements $u,v$.

\begin{theorem}
\label{thm:extension}
For $k\geq 2$, 
every poset $P$ of width $w$ without {\kk{k}} extends 
some interval order $Q$ of width at most $(2k-3)w$. 
\end{theorem}
\begin{proof} 
Let $C_{1}, \dots, C_{w}$ be a partition 
of $P$ into $w$ chains (which exists by Dilworth's theorem).
A subset $X$ of elements of $P$ will be called a {\DEF block} if 
$|X \cap C_{i}| \geq \min\{|C_{i}|, 2k-3\}$ and 
the elements in $X \cap C_{i}$ are consecutive
in the chain $C_{i}$ for every $i \in \{1, \dots, w\}$.  
Given a block $X$, the set $\up(X)$ is defined as the set 
of all elements $y$ of $P$ such that $y\in C_{i} - X$ for some 
$i \in \{1, \dots, w\}$ and $y > x$ for every $x \in X \cap C_{i}$.  
An element $u\in X$ is 
{\DEF good} if $u < v$ in $P$ for every $v \in \up(X)$. 

With a slight abuse of terminology, we say that an element $u$ of a set $X$ 
is {\DEF minimal in $X$} ({\DEF maximal in $X$})
if $u$ is a minimal (maximal, respectively) element of the poset induced by $X$. 

\begin{claim}
\label{claim:good}
If $X$ is a block with $\up(X) \neq \varnothing$, then
there is an index $i \in \{1, \dots, w\}$ such that
$\up(X) \cap C_{i} \neq \varnothing$  and the minimal element of $X\cap C_{i}$ is good.
\end{claim}
\begin{proof}
Reindexing the chains $C_{1}, \dots, C_{w}$ if necessary, we may assume
that there is an index $w'$ such that $\up(X) \cap C_{i} \neq \varnothing$
if and only if $i \in \{1, \dots, w'\}$. 
(Note that $w' \geq 1$, since otherwise $\up(X)$ would be empty.) 

Let $i\in \{1, \dots, w'\}$. Let $d_{i}$ be the minimal element in $\up(X) \cap C_{i}$. 
The set $(X \cap C_{i}) \cup\{d_{i}\}$ 
is a chain of size $2k-2 \geq k$; let $L_{i}$ 
be the set of the $k$ smallest elements in that chain,  
and let $a_{i}$ and $c_{i}$ be the minimal and maximal elements in $L_{i}$, respectively. 
Also let $b_{i}$ be the maximal element in $L_{i} - \{c_{i}\}$.  
Observe that $a_{i} \leq b_{i} < c_{i} \leq d_{i}$ in $P$. 
Let $U_{i} := (C_{i} \cap (X - L_{i})) \cup \{b_{i}, c_{i}, d_{i}\}$. 
Notice that $a_{i}=b_{i}$ and $c_{i} = d_{i}$ if $k=2$, while $a_{i} \neq b_{i}$ and $c_{i} \neq d_{i}$ if $k \geq 3$. In particular, we have $|U_{i}|=k$ in all cases. 

Define a directed graph $D$ with vertex set $V := \{1, \dots, w'\}$ as follows:  
For every $i,j\in V$, $i\neq j$, add an arc $(i, j)$ if $a_{i} \nless d_{j}$ in $P$. 
We claim that $D$ has no directed cycle. Arguing by contradiction, suppose 
$D$ has a directed cycle, and let $p$ denote its length. Reindexing the chains 
$C_{1}, \dots, C_{w'}$ if necessary, we may assume that the vertices of this cycle are
$1, 2, \dots, p$ (in order). 

Let us show that $c_{1} > b_{i}$ in $P$ for each $i \in \{2, \dots, p\}$, by induction on $i$. 
For the base case $i=2$, consider the two disjoint chains $L_{1}$ and $U_{2}$. Since
$|L_{1}| = |U_{2}| = k$ and since $P$ has no {\kk{k}}, some element $u \in L_{1}$ is comparable
to some element $v\in U_{2}$. Since $a_{1} \leq u$ and $v \leq d_{2}$ in $P$ but $a_{1} \nless d_{2}$ 
because of the arc $(1, 2)$ in $D$, we cannot have $u < v$. Thus $u > v$ in $P$, and since 
$c_{1} \geq u$ and $v \geq b_{2}$ this implies $c_{1} > b_{2}$. For the inductive step, 
assume $i \geq 3$. Since $c_{1} > b_{i-1}$, the set $L'_{i-1}:=\{c_{1}\} \cup (L_{i-1} - \{c_{i-1}\})$ is
a chain of $k$ elements. Since $L'_{i-1}$ and $U_{i}$ are disjoint, 
some element $u \in L'_{i-1}$ is comparable to some element $v\in U_{i}$. 
Similarly as before, we cannot have $u<v$ in $P$, because $a_{i-1} \leq u$ and $v \leq d_{i}$ 
but $a_{i-1} \nless d_{i}$. Hence $u > v$, implying $c_{1} > b_{i}$, as desired. 

Now, since $c_{1} > b_{p}$ in $P$, it follows that $d_{1} > a_{p}$. However, this 
contradicts the existence of the arc $(p, 1)$ in $D$. Thus $D$ has no directed cycle, as claimed.  

Since $D$ is acyclic, there exists a vertex $i \in V$ that has no outgoing arc. 
By the definition of $D$, 
this means that $a_{i} < d_{j}$ in $P$ for each $j \in \{1, \dots, w'\} - \{i\}$. 
Clearly, $a_{i} < d_{i}$ also holds. Therefore, 
$a_{i} < y$ for every $j \in \{1, \dots, w'\}$  
and every $y \in \up(X) \cap C_{j}$, implying that $a_{i}$ is good.  
\end{proof}

Define a sequence $B_{1}, \dots, B_{q}$ of blocks iteratively as follows: 
Let $B_{1}$ be the block obtained by taking the 
union of the $\min\{2k-3, |C_{i}|\}$ smallest elements in chain $C_{i}$ 
for every $i \in \{1, \dots, w\}$. For $j \geq 2$, if $\up(B_{j-1}) = \varnothing$ 
then we stop the process and $B_{j-1}=B_{q}$ becomes the last block of the sequence.   
Otherwise, let $u$ be a good element of $B_{j-1}$ as in Claim~\ref{claim:good}. Thus    
$u \in B_{j-1} \cap C_{i}$ for some $i \in \{1, \dots, w\}$ such that  
$\up(B_{j-1}) \cap C_{i} \neq \varnothing$.
Let $v$ be the smallest element of chain $\up(B_{j-1}) \cap C_{i}$, 
and let $B_{j} := (B_{j-1} - \{u\}) \cup \{v\}$.  

Observe that every element $u$ of $P$ appears in consecutive blocks of the
sequence $B_{1}, \dots, B_{q}$; let $I(u)$ be the closed interval $[i, j]$
of the real line where $i \leq j$ are indices such that $u$ is included
in precisely the blocks $B_{i}, B_{i+1}, \dots, B_{j}$ of the sequence. 
These intervals define an interval order $Q$ on the elements of $P$, where
$u < v$ in $Q$ if and only if $I(u) = [i,j]$ and $I(v)=[i',j']$ with $j < i'$. 
Every antichain $A$ of $Q$ corresponds to a set of pairwise intersecting intervals. 
By the Helly property of intervals, the latter intervals share a common point, which
implies that there is an index $i \in \{1, \dots, q\}$ such that $A \subseteq B_{i}$. 
Conversely, every block $B_{i}$ is an antichain of $Q$. 
It follows that the width of $Q$ is equal to $\max\{|B_{i}|: 1 \leq i \leq w\} \leq (2k-3)w$. 

Now, if $u < v$ in $Q$, then $j < i'$ where $I(u) = [i,j]$ and $I(v)=[i',j']$, 
and in particular $v \in \up(B_{j})$ by the definition of the blocks. 
Since $u \notin B_{j+1}$, it follows that $u$ is a good vertex of $B_{j}$, and
hence $u < y$ in $P$ for every $y \in \up(B_{j})$. In particular, $u < v$ in $P$.
Hence $P$ extends $Q$, and therefore $Q$ is an interval order as desired.  
\end{proof}

\begin{corollary}
\label{cor:pathwidth}
For $k \geq 2$, the incomparability graph $G$ of a poset $P$ of width $w$ without {\kk{k}} 
has pathwidth at most $(2k-3)w - 1$. 
\end{corollary}
\begin{proof}
Using Theorem~\ref{thm:extension}, let $Q$ be an interval order of width at most
$(2k-3)w$ such that $P$ extends $Q$. Then the incomparability graph $H$ of $Q$ is
an interval graph with $\omega(H) \leq (2k-3)w$ such that $G \subseteq H$. 
Therefore, $G$ has pathwidth at most $(2k-3)w -1$.  
(Note that the sequence $B_{1}, \dots, B_{q}$ of blocks defined in the proof 
of Theorem~\ref{thm:extension} provides a path decomposition of $G$ of width at most $(2k-3)w -1$.) 
\end{proof}

The poset consisting of $w$ pairwise incomparable chains
each of size $k-1$ has width $w$ and no {\kk{k}}, and its
incomparability graph is the complete $w$-partite graph with $k-1$
vertices in each color class, which has pathwidth $(k-1)(w-1)$ 
(see for instance~\cite[Lemma~8.2]{FW_EJC}).  
Thus, asymptotically, the 
bound in Corollary~\ref{cor:pathwidth} is within a factor of $2$ of optimal.

\subsection{Proof of upper bound} 

As mentioned in the introduction, 
Pemmaraju, Raman and Varadarajan~\cite{PRV_SODA} proved that 
First-Fit partitions every interval order of width $w$ into at most $10w$ 
chains, and this bound can be decreased to $8w$~\cite{PRV_TALG, NS_ORDER}. 
Thus $\FF(G) \leq 8\omega(G)$ for every interval graph $G$. 
While every graph $G$ with pathwidth $p$ is a spanning subgraph of an interval graph $H$ 
with $\omega(H) = p+1$, this does not immediately imply  
that $\FF(G) \leq 8(\pw(G) + 1)$. Indeed 
the invariant $\FF(G)$ is not monotone with respect to taking subgraphs;  
for instance, observe that $\FF(P_{4}) = 3 > 2 = \FF(C_{4})$. 
And to emphasize the point, $\FF(K_{n,n})=2$ but 
$\FF(K_{n,n} - M)=n$, where $M$ denotes a perfect matching of $K_{n,n}$.  
However, it turns out that the aforementioned upper bound on $\FF(G)$ 
in terms of the pathwidth of $G$ holds, as we now show.\footnote{
It should be noted that the invariant $\FF(G)$ on graphs $G$ of bounded pathwidth has been 
explicitly considered in~\cite{DMcC_IWPEC, DMcC_JCSS}. However it appears that the authors implicitly 
assumed that $\FF(G)$ is monotone with respect to subgraph inclusion when writing 
that upper bounds on $\FF(G)$ when $G$ is an interval graph with $\omega(G) \leq p+1$ 
immediately carry over to graphs $G$ with $\pw(G) \leq p$ (see~\cite[p.\ 22]{DMcC_IWPEC} 
and~\cite[p.\ 64]{DMcC_JCSS}).  
Note  that, by Theorem~\ref{thm:homomorphism},  
the maximum of $\FF(G)$ over all graphs $G$ with pathwidth at most $p$ is 
indeed always achieved by some interval graph $G$.  
}   

First, recall that a {\DEF homomorphism} from a graph $G$
to a graph $H$ is a function $f:V(G) \to V(H)$ that maps edges of $G$ to 
edges of $H$, that is,  
$f(u)f(v) \in E(H)$ for every edge $uv \in E(G)$. The graph
$G$ is said to be {\DEF homomorphic} to $H$ if such a mapping exists.

\begin{theorem}
\label{thm:homomorphism}
Every graph $G$ with pathwidth $p$ is homomorphic to an interval graph $H$ 
with $\omega(H) \leq p + 1$ and $\FF(G) \leq \FF(H)$. 
\end{theorem}
\begin{proof}
Consider a First-Fit coloring of $G$ with $c:= \FF(G)$ colors, and let $V_{1}, \dots, V_{c}$
denote the corresponding color classes (in order). 
Let $G'$ be an interval graph with $\omega(G') = p+1$ that is a spanning supergraph of $G$. 
Let $\mc{I}=\{I(v): v \in V(G')\}$ be an interval representation of $G'$, where
$I(v)$ denotes the interval corresponding to vertex $v$. 
For each $i \in \{1, \dots, c\}$, let $W_{i,1}, \dots, W_{i, n_{i}}$ denote the components 
of the graph $G'[V_{i}]$, and let further $I_{i,j}:= \cup\{I(v): v \in W_{i,j}\}$ for each
$j\in \{1, \dots, n_{i}\}$. 
Observe that $I_{i,j}$ is again an interval since $G'[W_{i,j}]$ is connected. 
Let $H$ be the interval graph defined by the latter intervals, and let $v_{i,j}$ denote
the vertex of $H$ corresponding to interval $I_{i,j}$, for each $i,j$ such that $1 \leq i \leq c$ 
and $1 \leq j \leq n_{i}$.

Let $f:V(G) \to V(H)$ be the function that maps each vertex $v$ of $G$ to vertex
$v_{i,j}$ of $H$ where $i,j$ is the unique pair of indices such that $v \in W_{i,j}$. 
Clearly $f$ is a homomorphism from $G$ to $H$.  
Also note that the mapping $f$ is surjective.  
Now consider an arbitrary clique $C$ of $H$.  
By the Helly property of intervals, there is a point $x$ on the real line 
that is contained in all the intervals corresponding to vertices in $C$.  
For each interval $I_{i,j}$ such that $v_{i,j} \in C$, 
there is at least one vertex in $W_{i,j}$ whose corresponding interval contains the point $x$. 
Thus choosing one such vertex of $G'$ for each vertex in $C$, we obtain a clique $C'$ of 
$G'$ with $|C'| = |C|$. It follows that $\omega(H) \leq \omega(G') = p+1$. 

Finally consider the coloring of $H$ with $c$ colors obtained by letting for $i=1, \dots, c$ 
the $i$-th color class be $Z_{i}:=\{v_{i,j}: 1 \leq j \leq n_{i}\}$. (Observe that this is 
a proper coloring of $H$ since $I_{i,j} \cap I_{i,j'} = \varnothing$ for $j\neq j'$.)
For each $i,j$ such that $1 \leq j < i$,  
each vertex $v\in V_{i}$ is adjacent in $G$ to some vertex $w\in V_{j}$, since 
$V_{1}, \dots, V_{c}$ is a First-Fit coloring of $G$, and thus
$f(v)$ is adjacent to $f(w)$ in $H$ (where $f$ is the homomorphism defined above). 
Since $f(v) \in Z_{i}$ and $f(w) \in Z_{j}$, 
it follows that every vertex in $Z_{i}$ has a neighbor in $Z_{j}$ in $H$ for
every $i,j$ such that $1 \leq j < i \leq c$, that is, 
$Z_{1}, \dots, Z_{c}$ is a First-Fit coloring of $H$. 
Hence $\FF(H) \geq c = \FF(G)$, and therefore $H$ is an interval graph
with the desired properties. 
\end{proof}

Theorem~\ref{thm:homomorphism} and the aforementioned bound of 
$\FF(G) \leq 8\omega(G)$ for interval graphs $G$ imply: 

\begin{corollary}
\label{cor:homomorphism}
$\FF(G) \leq 8(\pw(G) + 1)$ for every graph $G$. 
\end{corollary}

Let us remark that, as observed by an anonymous referee, 
a recent result of Kierstead and Saoub~\cite{KS_11} on $p$-tolerance graphs  
is similar to Theorem~\ref{thm:homomorphism} (see~\cite[Claim~9]{KS_11}), and is proved 
using a similar proof technique. Also, while it is not stated in these terms, 
we note that Claim~10 in~\cite{KS_11} shows that   
every $p$-tolerance graph $G$ has pathwidth at most 
$\lceil \frac{1}{1-p} \rceil \omega(G) - 1$. Combining this 
with Corollary~\ref{cor:homomorphism} yields $\FF(G) \leq 8\lceil \frac{1}{1-p} \rceil \omega(G)$, 
which is Theorem~2 from~\cite{KS_11}.  

Since a First-Fit chain partition of a poset $P$ can equivalently be
seen as a First-Fit coloring of the incomparability graph of $P$, 
Corollary~\ref{cor:pathwidth} and Corollary~\ref{cor:homomorphism} together
imply the follow result. 

\begin{theorem}
\label{thm:main}
For $k \geq 2$, 
First-Fit partitions every poset $P$ of width $w$ without {\kk{k}} into 
at most $8(2k-3)w$ chains.
\end{theorem}

\subsection{Proof of lower bound}

The bound in Theorem~\ref{thm:main} is best possible up to a constant factor, as we now show. 

\begin{theorem}
\label{thm:lower_bound}
For every $k\geq 2$ and $w \geq 2$, there exists a 
poset $P$ of width $w$ without {\kk{k}} on which 
First-Fit can be forced to use at least $(k-1)(w-1)$ chains. 
\end{theorem}
\begin{proof}
For $w=2$ this can be shown using a construction due to Kierstead~\cite{K_CM}: 
For $q \geq 2$, define $P_{q}$ as the poset on the
set of elements $V_{q} := \{v_{1,1}, v_{2, 1},  v_{2, 2}, 
v_{3, 1},  v_{3, 2}, v_{3,3}, \mathbf{\dots}, v_{q, 1}, v_{q, 2}, \dots, v_{q, q}\}$, 
where $v_{i,j} < v_{i',j'}$ in $P_{q}$ if and only if 
$i \leq i' - 2$, or $i\in \{i' -1, i'\}$ and  $j \leq j' - 1$. 
See Figure~\ref{fig:P5} for an illustration. 

\fig{0.35}{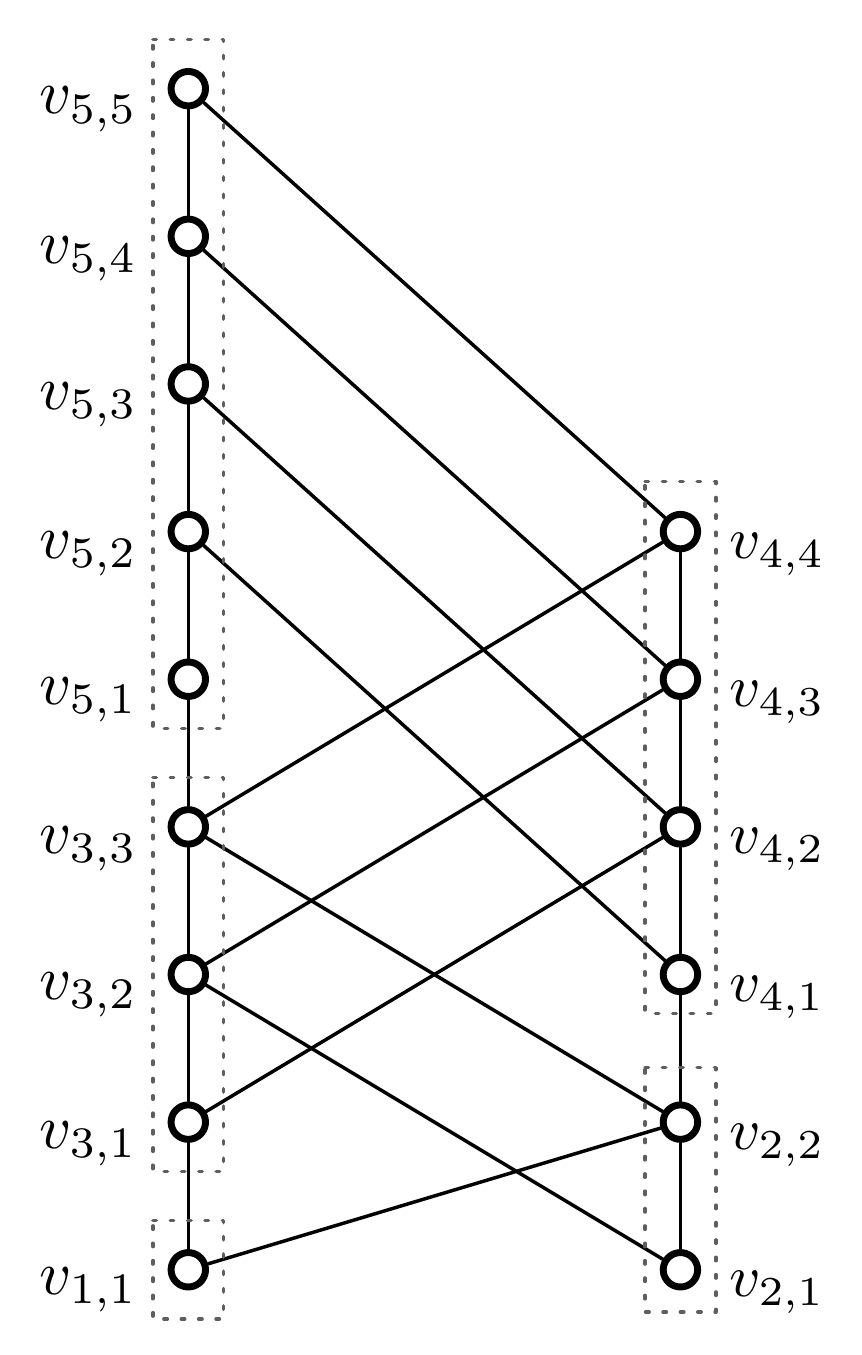}{The Hasse diagram of poset $P_{5}$. 
The dotted rectangles represent the chains $\{v_{i,j}: 1 \leq j \leq i\}$ for 
each $i$.}{fig:P5}

The poset $P_{q}$ has width $2$, because it has two incomparable elements and 
$V_{q}$ can be partitioned into the two chains $\{v_{i,j}: 1 \leq j \leq i\leq q, i \textrm{ odd}\}$
and $\{v_{i,j}: 1 \leq j \leq i \leq q, i \textrm{ even}\}$. 
Consider the ordering of the elements of $P_{q}$ suggested by their indices, namely
$v_{1,1}, v_{2, 1},  v_{2, 2},  \mathbf{\dots}, v_{q, 1}, v_{q, 2}, \dots, v_{q, q}$,  
which we call the {\DEF natural ordering}.  
Given this ordering, observe that First-Fit assigns element $v_{i,j}$ to the $(i-j+1)$-th chain. 
Hence First-Fit uses exactly $q$ chains in total (as proved by Kierstead~\cite{K_CM}).   
Also, every element of $P_{q}$ is incomparable to at most $q$ others, 
implying that $P_{q}$ has no {\kk{(q+1)}}. Therefore, $P_{q}$ with $q=k-1$ is a
poset with the desired properties.  

Now assume $w \geq 3$. We modify Kierstead's construction as follows:  
Take the disjoint union of $w-1$ copies of $P_{k-1}$. Denote by $v^{\l}_{i,j}$ the
element $v_{i,j}$ in the $\l$-th copy of $P_{k-1}$, and add the following
comparisons between elements from distinct copies: For $1 \leq \l < \l' \leq w-1$ and  $i \neq k-1$, 
we have $v^{\l}_{i,j} < v^{\l'}_{i',j'}$. 
Let $Q_{k,w}$ denote the resulting poset. See Figure~\ref{fig:Q} for an illustration. 

\fig{0.6}{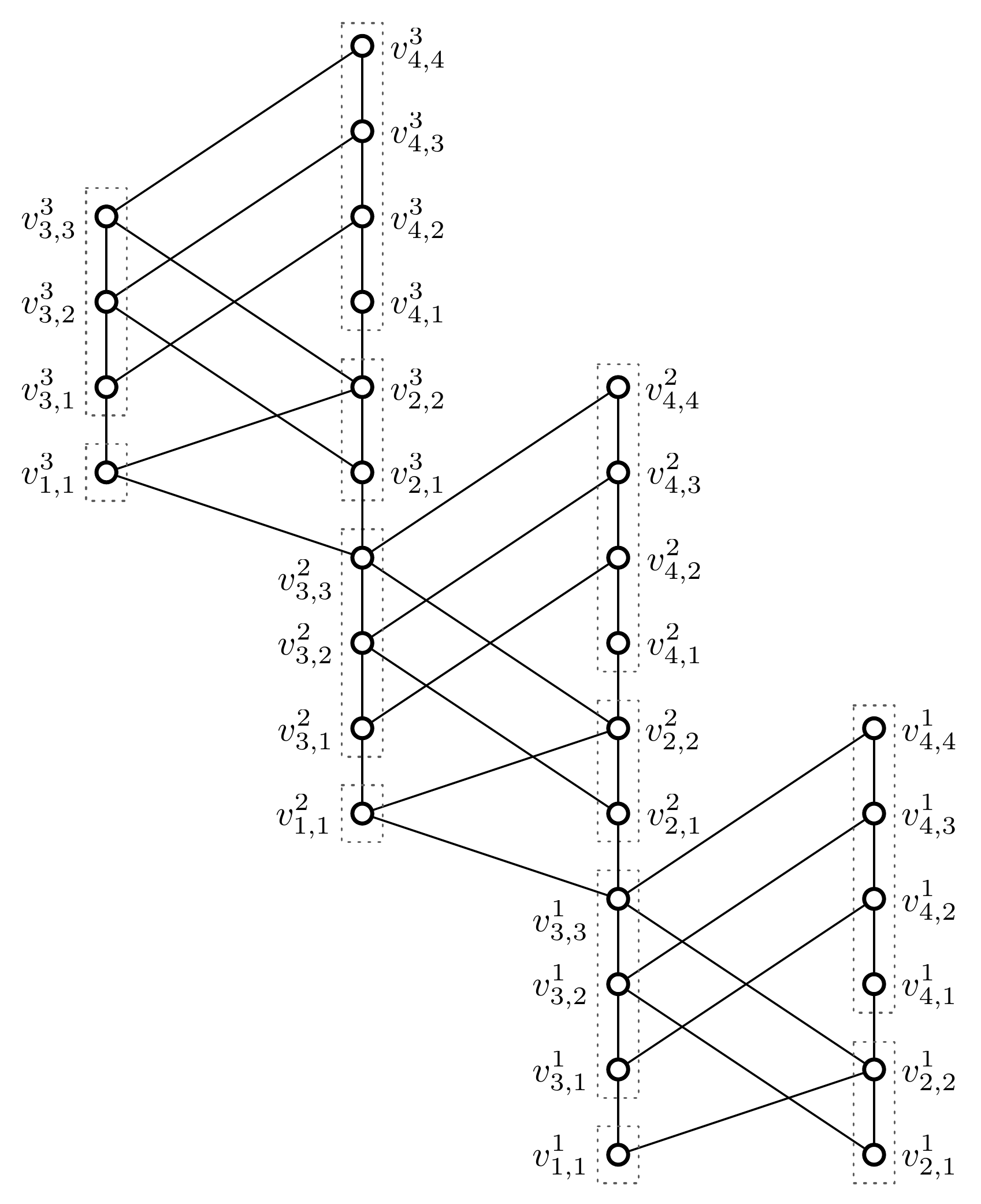}{The Hasse diagram of poset $Q_{5,4}$. The copies of $P_{4}$ are ordered
from bottom-right to top-left.}{fig:Q}

The poset $Q_{k,w}$ has width exactly $w$. 
We claim that $Q_{k,w}$ has no {\kk{k}}.  
Arguing by contradiction, suppose that $A, B$ are two disjoint chains in $Q_{k,w}$
of size $k$ that are incomparable.  
Since the set $X := \{v_{k-1,j}^{\l}: 1\leq j \leq k-1, 1 \leq \l \leq w-1\}$ induces
a poset of height $k-1$, the sets $A - X$ and $B - X$ are not empty. 
Moreover, every $u\in A - X$ and $v \in B - X$ belong to the same copy of $P_{k-1}$, 
as otherwise they would be comparable; thus $A-X$ and $B-X$ are both subsets of the $\l$-th
copy of $P_{k-1}$ for some $\l \in \{1, \dots, w-1\}$. 
Recall that $P_{k-1}$ has no {\kk{k}}, 
implying that at least one of $A, B$, say $A$, 
has an element $x$ that belongs to another copy of $P_{k-1}$, 
say the $\l'$-th one. Thus $x \in X$. 
Let $u\in A - X$ and $v \in B - X$. 
If $\l' < \l$, then $u$ and $x$ are incomparable, contradicting $u, x \in A$. 
If $\l' > \l$, then $x > v$ in $Q_{k,w}$, contradicting the fact that $x, v$ are incomparable. 
Both cases leading to a contradiction, we deduce that the two chains $A, B$ do not exist. 
 
Now, given the ordering
of the elements of $Q_{k,w}$ obtained by concatenating the natural orderings of
the $w-1$ copies of $P_{k-1}$ in order, First-Fit assigns element $v^{\l}_{i,j}$ 
to the $((k-1)(\l-1) + (i-j+1))$-th chain, as is easily checked. 
Hence First-Fit uses $(k-1)(w-1)$ chains in total. 
\end{proof}

\section*{Acknowledgments}
We are grateful to an anonymous 
referee for her/his helpful comments, 
and in particular for pointing out that the bound in Theorem~\ref{thm:extension} could 
be improved from our original bound of $2kw$  to $(2k-3)w$. 

\bibliographystyle{abbrv}
\bibliography{paper}

\end{document}